\documentclass{article}

\usepackage{amsfonts}
\usepackage{amssymb}
\usepackage{amsmath}
\usepackage{latexsym}
\usepackage{amsthm}
\usepackage{graphicx}
\usepackage{pstricks}

\newtheorem{thm}{Theorem}[section]
\newtheorem{cor}[thm]{Corollary}
\newtheorem{definition}[thm]{Definition}
\newtheorem{conjecture}[thm]{Conjecture}
\newtheorem{lemma}[thm]{Lemma}
\title{Some inequalities for the Tutte polynomial}
\author{Laura E. Chavez-Lomel\'{\i}\thanks{Universidad Autonoma
    Metropolitana, Unidad Azcapotzalco,  Avenida San Pablo  180,
    colonia Reynosa Tamaulipas, Delegaci\'on
    Azcapotzalco. e-mail:lelc@correo.azc.uam.mx},
      Criel Merino\thanks{Instituto de Matem\'aticas, Universidad
      Nacional Aut\'onoma de M\'exico, Area de la Investigaci\'on
      Cient\'{\i}fica, Circuito Exterior, C.U. Coyoac\'an 04510,
      M\'exico,D.F. M\'exico. e-mail:merino@matem.unam.mx. Supported
      by Conacyt of M\'exico Proyect 83977},
       Steven D. Noble\thanks{Department of Mathematical Sciences,
Brunel University, Kingston Lane, Uxbridge UB8 3PH,
U.K. e-mail:steven.noble@brunel.ac.uk},
      \and Marcelino Ram\'{\i}rez-Iba\~nez\thanks{Instituto de Matem\'aticas, Universidad
      Nacional Aut\'onoma de M\'exico, Area de la Investigaci\'on
      Cient\'{\i}fica, Circuito Exterior, C.U. Coyoac\'an 04510,
      M\'exico,D.F. M\'exico. e-mail: marchelino@gmail.com}
     }
\date{\today}

\begin{document}
\maketitle
\thispagestyle{empty}

%\noindent Running title: inequalities for the Tutte polynomial

%\noindent Criel Merino\\
%IMATE OAXACA \\
%Le\'on 2, altos, colonia centro, \\
%C.P. 68000, Oaxaca de Ju\'arez, M\'exico.\\
% e-mail:merino@matem.unam.mx.
%\newpage

\begin{abstract}

 We prove that  the Tutte polynomial of a coloopless paving matroid is
 convex along the portions of the line segments $x+y=p$ lying in the
 positive quadrant. Every coloopless paving matroids is in the class
 of
 matroids which contain two disjoint bases or whose ground
 set is the union of two bases of  $M^*$. For this latter class we
 give a
 proof that $T_{M}(a,a)\leq \max\{T_{M}(2a,0),T_{M}(0,2a)\}$ for
 $a\geq 2$. We conjecture that  $T_{M}(1,1)\leq
 \max\{T_{M}(2,0),T_{M}(0,2)\}$ for the same class of matroids. We
 also
 prove this conjecture for some families of graphs and
 matroids.
 \end{abstract}

\section{Introduction}
\label{sec:intro}

 The Tutte polynomial is a two variable polynomial which can be
 defined for  a graph $G$ or, more generally, a matroid $M$. The Tutte
 polynomial  has many interesting combinatorial interpretations
 when evaluated at different points ($x$, $y$) and along several
 algebraic
 curves. For example, for a graph $G$, the Tutte polynomial along the
 line $y=0$ is the chromatic polynomial, after a suitable change of
 variable and multiplication by an easy term. Similarly, we can
 get the flow polynomial of a graph and the all terminal reliability
 of a network and the partition function of the $Q$-state Potts
 model. When considering  a GF($q$)-representable matroid, the Tutte
 polynomial gives us the weight enumerator of  linear codes over
 GF($q$) associated to $M$. All the necessary background on the Tutte
 polynomial is contained in Section~\ref{sec:prelim}.

 It is well-known~\cite{Bjorner} that the Tutte polynomial of a matroid $M$
 has an expansion
\[ T_M(x,y) = \sum_{i,j} t_{ij}x^iy^j,\]
in which each coefficient $t_{ij}$ is non-negative. Consequently, for
$m\geq 0$ and for any $b$, if $(x,y)$ lies on the portion of the line
$y=mx+b$ lying in the positive quadrant, then $T_M$ increases as $x$
increases. The simplicity of the behaviour of $T$ along lines with
positive gradient suggests the study of the behaviour of $T_M$ along
lines with negative gradient in the positive quadrant. Merino and
Welsh~\cite{MrWl} were the first to consider this and were particularly
interested in resolving the question of whether the Tutte polynomial
is convex along the portion of the line $x+y=2$ lying in the positive
quadrant. They made the following intriguing conjecture.

\begin{conjecture}\label{old}
Let $G$ be a 2-connected graph with no loops. Then
\begin{equation}
\max\{T_G(2,0),T_G(0,2)\} \geq T_G(1,1). \label{eq:MW}
\end{equation}
\end{conjecture}
Any graph with at least one loop and at least one isthmus fails to
satisfy~\eqref{eq:MW}, so~\eqref{eq:MW} cannot hold for all
graphs. The main reason for the particular interest in the points
$(2,0)$, $(0,2)$ and $(1,1)$ is that in a connected graph $G$,
$T_G(2,0)$, $T_G(0,2)$ and $T_G(1,1)$ give the number of acyclic
orientations, totally cyclic orientations and spanning trees in
$G$. Definitions of acyclic and totally cyclic orientations are
contained in Section~\ref{sec:prelim}.

%% Do we want to mention the extension to matroids here? (i.e. your
%% conjecture with %% Conde?)

A related question is to determine whether any loopless 2-connected
graph $G$ satisfies the apparently stronger requirement
\[ T_G(2,0)T_G(0,2) \geq (T_G(1,1))^2.\]
Relatively little progress has been made to resolve these
questions. However, Jackson in~\cite{Jackson} has shown, with a
clever argument, that for any connected matroid $M$,
\[T_G(3,0)T_G(0,3) \geq (T_G(1,1))^2.\]

In this paper we make three contributions. First, in Section 4, we
show that the Tutte polynomial $T$ of a coloopless paving matroid
satisfies the inequality
\begin{equation}
tT(x_1,y_1)+(1-t)T(x_2,y_2) \geq T(tx_1+(1-t)x_2,ty_1+(1-t)y_2), \label{eq:convex}
\end{equation}
where $0\leq t \leq 1$ and $x_1,x_2,y_1,y_2$ are non-negative and
satisfy $x_1+y_1=x_2+y_2$. That is, $T$ is convex along the portions
of the line segments $x+y=p$ lying in the positive quadrant.
A paving matroid is one in which all circuits have size at least
$r(M)$. Interest in them stems from a conjecture in~\cite{MyNwWlWh}
which says that asymptotically almost every matroids is paving.
The special case of~\eqref{eq:convex}, obtained by setting
$x_1=y_2=2$, $x_2=y_1=0$ and $t=1/2$, establishes \eqref{eq:MW} for
the class of paving matroids. Therefore if the above conjecture is true
then we have established~\eqref{eq:MW} for, asymptotically, almost all
coloopless matroids.

Second, in Section 5, we prove that \eqref{eq:MW} holds for some
smaller classes of matroids and graphs that are not paving
matroids. Finally, in Section 3, we prove that if the ground set of
$M$ contains two disjoint bases then $T_M(0,2a) \geq T_M(a,a)$ and
dually if the ground set of $M$ is the union of two bases then
$T_M(2a,0)\geq T_M(a,a)$. These results cannot be obtained with the
methods used by Jackson in~\cite{Jackson}.

We conclude with a brief discussion of the natural question of for
which matroids is $T_M$ a convex function in the positive quadrant?

%%%%%%%%%%%%%%%%%%%%%%%%%%%%%%%%%%%%%%%%%%%%%%%%
%%%%  PRELIMINARIES  %%%%%%%%%%%%%%%%%%%%%%%%%%%
%%%%%%%%%%%%%%%%%%%%%%%%%%%%%%%%%%%%%%%%%%%%%%%%

\section{Preliminaries}
\label{sec:prelim}

 We assume that the reader has some familiarity with matroid and graph
 theory. For matroid theory we follow Oxley's book \cite{Oxley} and
 for graph
 theory we follow Diestel's book \cite{Diestel}.

%\subsection{The Tutte polynomial}
%\label{sec:tutte}

 The Tutte polynomial is a matroid invariant over the ring
 $\mathbb{Z}[x,y]$. Further details of many of the concepts
  treated here can be found in Welsh \cite{Welsh} and Oxley and
 Brylawski \cite{BrOx}.

 Some of the richness of the Tutte polynomial is due to its numerous
 equivalent definitions. One of the simplest
 definitions, which is often the easiest way to prove properties of
 the Tutte polynomial, uses the notion of rank.

 If $M=(E,r)$ is a  matroid, where $r$ is the rank-function of $M$,
 and $A\subseteq E$, we denote $r(E)-r(A)$ by $z(A)$  and $|A|-r(A)$ by
 $n(A)$.

\begin{definition}\label{definition}
 The Tutte polynomial of $M$, $T_{M}(x, y)$, is defined as follows:
 \begin{equation}\label{eq:expansion}
     T_{M}(x, y) = \sum_{A\subseteq E} (x-1)^{z(A)}(y-1)^{n(A)}\;.
 \end{equation}
\end{definition}

 Almost immediately  we see that $T_{M}(1,1)$ equals the number of
 bases of $M$ and $T_{M}(2,2)$ equals $2^{|E|}$.
 Recall that if $M=(E,r)$ is a matroid, then $M^{*}=(E,r^{*})$ is its
 dual matroid, where $r^{*}(A)=|A|-r(E)+r(E\setminus A)$. Because
 $z_{M^{*}}(A)=n_{M}(E\setminus A)$ and $n_{M^{*}}(A)=z_{M}(E\setminus
 A)$ it follows that $T_{M}(x,y)=T_{M^{*}}(y,x)$.

 For a graphic matroid $M(G)$, the evaluations of the Tutte
 polynomial at $(2,0)$ and $(0,2)$ equal
 the number of acyclic orientations and the number of totally cyclic
 orientations of $G$, respectively. An acyclic orientation of a graph
 $G$ is an orientation where there are no directed cycles. A totally
 cyclic orientation is an orientation where every edge is in a
 directed cycle. See~\cite{BrOx} for a proof of this result. In this
 situation we let $\alpha(G)$ and $\alpha^{*}(G)$ denote $T_{G}(2,0)$ and
 $T_{G}(0,2)$ respectively. If $G$ is connected, the
 number of spanning trees of $G$ is the evaluation of the Tutte
 polynomial at $(1,1)$ and this quantity is denoted by  $\tau(G)$.

 The Tutte polynomial may be also defined by a linear recursion relation
 given by deleting and contracting elements that are neither loops nor
 isthmuses.

\begin{definition}\label{def:recursive}
If $M$ is a matroid, and $e$ is an element  that is neither an isthmus
nor a loop, then
\begin{equation}\label{eq:deletion-contraction}
   T_{M}(x,y) = T_{M\setminus e}(x, y) + T_{M/e}(x, y).
\end{equation}  \label{end_recursion}
 If there is no such element $e$, then  $T_{M}(x, y) = x^{i}y^{j}$
  where $i$ and $j$ are the number of isthmuses and loops of $M$ respectively.
\end{definition}

 The proof that Definition~\ref{definition} and~\ref{def:recursive} are
 equivalent can be found in~\cite{BrOx}. We still require another
 (equivalent) definition of the Tutte polynomial but first we
 introduce the relevant notions.

 Let us fix an ordering $\prec$ on the elements of $M$, say
 $E=\{e_1,\ldots,e_m\}$, where $e_i\prec e_j$ if $i<j$. Given a fixed
 basis $S$, an element $e$ is called \emph{internally active} if $e\in S$
 and it is the smallest edge with respect to $\prec$ in the only
 cocircuit disjoint from $S\setminus
 \{e\}$. Dually, an element $f$ is \emph{externally active} if
 $f\not\in S$ and
 it is the smallest element in the only circuit contained in $S\cup
 \{f\}$.  We
 define $t_{ij}$ to be the number of bases with $i$
 internally activity elements and $j$ externally activity
 elements.  In \cite{Tutte} Tutte defined $T_{M}$ using these
 concepts. A proof of the equivalence with Definition~\ref{definition}
 can be found in~\cite{Bjorner}.

 \begin{definition} \label{def:activities}
 If $M=(E,r)$ is a matroid with a total order on its ground set, then
\begin{equation}\label{tutte_activities}
   T_{M}(x,y)=\sum_{i,j}t_{ij}x^iy^j\;.
\end{equation}
In particular, the coefficients $t_{ij}$ are independent of the total order
used on the ground set.
\end{definition}

 By an inductive argument using equation
 (\ref{eq:deletion-contraction}), it can
 be proved that $t_{10} = t_{01}$ when $E(M)\geq 2$. This is one of
 a  number of identities known to hold for the coefficients $t_{ij}$.
 For a complete characterization of all the
 affine linear relations that hold among the coefficients  $t_{ij}$
 see Theorem 6.2.13 in~\cite{BrOx}. From there we extract the
 relations that we need.

 \begin{thm}\label{relations}
  If  a rank-$r$ matroid $M$ with $m$ elements has neither loops nor
  isthmuses, then
     \begin{enumerate}
     \item $t_{ij}=0$, whenever $i>r$ or $j>m-r$;
     \item $t_{r0}=1$ and $t_{0, m-r}=1$;\label{max_rank_nullity}
     \item  $t_{rj}=0$ for all $j>0$ and $t_{i,m-r}=0$ for all $i>0$.
     \end{enumerate}
 \end{thm}
  The previous result follows easily from
  Definition~\ref{def:activities}. In~\cite{BrOx} the statement is for
  simple matroids (geometries) but
  it is easy to extend it to matroids with parallel elements.

%%%%%%%%%%%%%%%%%%%%%%%%%%%%%%%%%%%%%%%%%%%%%%%%%%%
%%%%%  first INEQUALITIES %%%%%%%%%%%%%%%%%%%%%%%%%
%%%%%%%%%%%%%%%%%%%%%%%%%%%%%%%%%%%%%%%%%%%%%%%%%%%

\section{Some inequalities for the Tutte polynomial}
\label{sec:ineq}
 From the  results in the previous section it is easy to prove the
 following result stated in~\cite{MrIbRd}.
\begin{thm}\label{basic_2_2}
 If a matroid $M$ has neither loops nor  isthmuses, then
 \begin{equation*}
\max\{T_{M}(4,0),T_{M}(0,4)\}\geq T_{M}(2,2).
\end{equation*}
\end{thm}
\begin{proof}
 Let $r$ be the rank  and $m$ the number of elements of $M$.
\begin{align*}
  \max\{T_{M}(4,0),T_{M}(0,4)\}&\geq \max\{4^{r},4^{m-r}\}\\
                               &=\max\{2^{2r},2^{2(m-r)}\}\\
                               &\geq 2^{m}=T_{M}(2,2)\;,
\end{align*}
 where the first inequality follows from equation~(\ref{tutte_activities})
 combined with~\ref{max_rank_nullity} in Theorem~\ref{relations}.
\end{proof}

Note that, for a matroid $M=(E,r)$ with dual
$M^{*}=(E,r^{*})$, the following inequalities are equivalent for any
$A\subseteq E$.
\begin{eqnarray}
   |A| &\leq& |E|-2(r(E)-r(A)),\label{2-bases} \\
  |E\setminus A| &\leq& 2r^{*}(E\setminus A)\ \text{and}\label{dual_2-bases}\\
   z(A)+n(A) &\leq&  |E|-r.
\end{eqnarray}

%% \begin{proof}
%%   By hypothesis  $n(A)=|A|-r(A)\leq |E|-2r(E)+r(A)$. Then add
%%   $z(A)=r(E)-r(A)$ to the previous inequality to obtain the result.
%% \end{proof}

 We now restrict attention to matroids $M$ in which all
 subsets $A$ of the ground set $E$ satisfy the (equivalent)
 inequalities above.  By a classical result of
 J.~Edmonds~\cite{Edmonds2}, these
 are the matroids that contain two disjoint bases; by duality, these are
 the matroids $M$ whose ground set is the union of two bases of
 $M^*$.

  As every term $(x-1)^{z(A)}(y-1)^{n(A)}$ in $T_{M}$ has
  $x^{z(A)}y^{n(A)}$ as its monomial of maximum degree, the following
  theorem  follows directly from the set of inequalities above.

\begin{thm}\label{thm:max_degree}
   If a matroid $M$ contains two  disjoint bases,  then $t_{ij}=0$, for
  all $i$ and $j$ such that
  $i+j> m-r$. Dually, if its ground set is the union of two bases, then
  $t_{ij}=0$, for all $i$ and   $j$ such that $i+j>r$.
\end{thm}

 Now, it is easy to prove an infinite set of inequalities for the
 Tutte polynomial of a matroid that contains two disjoint bases or
 whose ground set is the union of two bases. This theorem was stated
 in~\cite{MrIbRd}.

\begin{thm}\label{main}
 If a matroid $M$ contains two disjoint bases, then
 \begin{equation}
   \label{eq:main}
   T_{M}(0,2a)\geq T_{M}(a,a)\;,
 \end{equation}
   for all $a\geq 2$. Dually, if its ground set is the union of two
   bases, then
\begin{equation}
   \label{eq:main2}
    T_{M}(2a,0)\geq T_{M}(a,a),
 \end{equation}
  for all $a\geq 2$.
 \end{thm}
 \begin{proof}
    Let us  consider just the case when $M$ has two disjoint bases: the
   other case follows from duality. In this situation $m-r\geq r$. From
   the proof of Theorem~\ref{basic_2_2} and
   equation~(\ref{tutte_activities}) we  have
   $4^{m-r}\geq T_{M}(2,2)=\sum_{i,j}t_{ij}2^{i+j}$. Multiplying this
   inequality by $(a/2)^{m-r}$ we get
\[
  (2a)^{m-r}\geq \sum_{i,j}t_{ij} \left( \frac{a}{2}\right)^{m-r} 2^{i+j}
            \geq \sum_{i,j}t_{ij} \left( \frac{a}{2}\right)^{i+j}2^{i+j}
               = \sum_{i,j}t_{ij} a^{i+j}.
\]
 The second inequality follows from Theorem~\ref{thm:max_degree}. Thus
\[T_{M}(0,2a)\geq (2a)^{m-r} \geq  \sum_{i,j}t_{ij}a^{i+j} = T_{M}(a,a). \]
\end{proof}

 We can sum up the previous result by saying that if $M$ contains two
 disjoint bases or  its ground set is the union of two bases then
  \begin{equation}\label{main_inequality}
    \max\{T_{M}(2a,0),T_{M}(0,2a)\}\geq T_{M}(a,a),
  \end{equation}
 for $a\geq 2$.
 Some classes of matroids which contain two disjoint bases or whose
 ground set is the union of two bases are mentioned in the following

 \begin{cor}\label{list_matroids}
 For a matroid $M$, we have that $T_{M}$ satisfies (\ref{main_inequality}),
 for all $a\geq 2$ whenever $M$ is  one of the following:
  \begin{itemize}
  \item an identically self-dual matroid $M$,
%  \item the  uniform matroid $U_{r,n}$ for $0\leq r\leq n$.
  \item a rank-$r$ projective geometry over GF($q$) or its
  dual, for $r\geq 2$.
   \end{itemize}
\end{cor}
  \begin{proof}
    A matroid $M=(E,r)$ is  identically self-dual  if $M=M^{*}$,
     so,  $B$ is a basis of $M$ if and only if $E-B$ is a basis of
     $M$.

%    The matroid $U_{r,n}$ has two disjoint bases if $2r\leq n$ and its
%    ground set is the union of two bases if $n\leq 2r$.

    The matroid PG($r$, $q$) contains the graphic matroid $W_{r+1}$,
    the $r+1$-wheel, as a submatroid for
    $r\geq 3$, see~\cite{Oxley}. The latter  contains two disjoint bases. Thus,
    PG($r$, $q$) contains two disjoint bases. A projective plane of
    order $m\geq 4$ contains $U_{2,4}\oplus_2 U_{2,4}$ as a
    submatroid. Again, the latter  contains two
    disjoint bases. Thus, such a projective plane
    contains two disjoint bases. The only projective plane of order 3
    is the Fano matroid which clearly contains two disjoint bases.
  \end{proof}

There are more classes of matroids that can be added to the previous
list, for instance, coloopless paving matroids. However, in the next
section we will prove a much stronger result for them. The graphic
matroids corresponding to the families of graphs in our next result
may also be added to the list.
%%   There are more classes of matroids we can add to the previous list,
%%   for example coloopless paving matroids. In the following section we
%%   establish this result.
%%   For graphic matroids we have the following
%%  Observe that  in general it is  false that $t_{ij}=0$ if $i+j\geq
%%  \max\{ r,m-r\}$
%%  for a simple bridgeless graph $G$ with Tutte polynomial
%%  $T_{G}(x,y)=\sum_{i,j}t_{ij}x^iy^j$.

  \begin{cor}\label{list_graphs}
    For a graph $G$, $T_{G}$ satisfies
    (\ref{main_inequality}),
    for all $a\geq 2$ whenever $G$ is  one of the following:
    \begin{itemize}
    \item a  4-edge-connected graph,
    \item a  2-connected threshold graph,
    \item a complete bipartite  graph,
    \item a series-parallel graph,
    \item a  3-regular graph,
    \item a bipartite planar graph,
    \item a Laman graph,
    \item a triangulation,
    \item  the wheel graph $W_n$, for $n\geq 2$,
    \item  the square lattice $L_n$, for $n\geq 2$,
    \item  the $n$-cycle $n\geq 2$,
    \item  a tree with  $n$ edges, for $n\geq 1$.
    \end{itemize}
    \end{cor}
  \begin{proof}
   By the classical result in \cite{Tut61} every 4-edge-connected
  graph has two edge-disjoint
  spanning trees. It is easy to see that 2-connected threshold and
  wheel graphs have two edge-disjoint spanning trees. Using the
  expression for computing the arboricity of a graph given in
  \cite{NW64} we get that series-parallel,
  3-regular, bipartite planar, and Laman graphs all have arboricity
  two,  which is equivalent to having two spanning trees that cover all the
  edges of the graph. Triangulations are geometric duals of 3-regular
  planar graphs, so they have two edge-disjoint spanning trees.

   It is easy to see that each of $K_{2,m}$ for $m\geq 2$, $K_{3,3}$, the
   square lattice $L_n$ for $n\geq 2$, the $n$-cycle for $n\geq 2$, and a
   tree have two spanning trees which cover all the edges in the
   graph. With the exception of the case $n=m=3$, if both $n$ and $m$
   are at least 3, then $K_{n,m}$ always has two edge-disjoint
   spanning trees.
\end{proof}

%%%%%%%%%%%%%%%%%%%%%%%%%%%%%%%%%%%%%%%%%%%%%%%%%%%%%%
%% PAVING MATROIDS%%%%%%%%%%%%%%%%%%%%%%%%%%%%%%%%%%%%
%%%%%%%%%%%%%%%%%%%%%%%%%%%%%%%%%%%%%%%%%%%%%%%%%%%%%%

\section{Paving matroids}\label{sec:paving}

 A paving matroid $M=(E,r)$ is a matroid whose circuits all have size
 at least $r$. Paving matroids are closed under minors and the set of
 excluded minors
 for the class consists of the matroid $U_{2,2}\oplus U_{0,1}$, see for
 example~\cite{IbMrNb}.  The interest about paving matroids goes back
 to 1976 when Dominic Welsh
 ask if  most matroids are paving, see~\cite{Oxley}. More recently,
 the authors
 in~\cite{MyNwWlWh} pose as a conjecture that asymptotically almost
 every matroid is paving.

 First, we prove that most paving matroids either contain two disjoint
 bases or  their ground set is the union of two bases. Consequently
 paving matroids fall within the class of matroids considered in the
 previous section.

 \begin{thm}\label{paving_DB_or_UI}
  Let $M=(E,r)$ be a rank-$r$ paving matroid with $n$ elements,
  \begin{itemize}
  \item if $2r>n$, then $E$ is the union of two bases,
  \item if $2r\leq n$ and $M$ is coloopless, then $M$ contains two
  disjoint bases.
  \end{itemize}
 \end{thm}
 \begin{proof}
  In the first case, take $B_1$ to be a basis of $M$, then
  $I_2=E\setminus B_1$ has size $n-r<r$, so it is independent and we
  can extend it to a basis $B_2$. Thus $E=B_1\cup B_2$.

 In the second case, if $M$ has a circuit $C$ of size $r+1$, then
 $C'=E\setminus C$ has size $n-r-1\geq r-1$. Let  $I$ be a set
 of size $r-1$ contained in $C'$. As $I$ is independent and $C$ is
 spanning, there exists $a\in C\setminus I$ such that $I\cup\{a\}$ is
 a basis. But $C\setminus\{a\}$ is also a basis. Thus, we have two
 disjoint bases.

 Let $M$ be a coloopless paving matroid with no circuits of size $r+1$
 and suppose that $2r\leq n$. Let  $B$ be a basis of $M$. Then
 either $E\setminus B$ contains a basis, in which case we have
 finished the
 proof, or $r(E\setminus B)=r-1$. In the latter case, let $H$ be  the
 hyperplane
 defined as the closure of $E\setminus B$, and
 $I=E\setminus  H\subseteq B$. The set $I$ has size $p+1$ with $p\geq
 1$ as $M$ is  coloopless.

 We show that in this case $M$ also has two disjoint bases. Let
 $I'=I\setminus\{a\}$, for some $a\in I$. Then, $I'$ is a
 non-empty independent set of size $p$ with the property that for any
 circuit $C$ of size $r$ contained in $H$, $I'\cup C$ contains a basis
 of $M$. Thus, there is a basis $B_1$ of $M$ of the form $I'\cup A_1$
 for some subset $A_1$ of $H$ of size $r-p$.  Now, let
 $B_2=\{a\}\cup A_2$ for some $A_2\subseteq H\setminus A_1$ of size
 $r-1$. This is possible as $|H\setminus A_1|=(n-p-1)-(r-p)
 =(n-r)-1\geq r-1$. Thus, $B_1$ and $B_2$ are disjoint bases of $M$.
 \end{proof}

The main goal of this section is prove that for any coloopless paving matroid
\begin{equation}
tT(x_1,y_1)+(1-t)T(x_2,y_2) \geq T(tx_1+(1-t)x_2,ty_1+(1-t)y_2),
\end{equation}
whenever $0\leq t \leq 1$ and $x_1,x_2,y_1,y_2$ are non-negative and
satisfy $x_1+y_1=x_2+y_2$.  Notice
that this inequality is a much stronger statement than~(\ref{eq:main2})
as it says
that $T$ is a convex function along the portions of the line $x+y=p$
lying in the positive quadrant, rather than merely saying that the
value of $T$ at one of the endpoints of the line segment is greater
than the value of $T$ at its midpoint.

 Our main tools for establishing the convexity of $T$ are the
 following easy results.
 \begin{lemma}\label{convexity_duality}
  Let $M$ be a matroid. Either, both $T_M(x,y)$ and $T_{M^{*}}(x,y)$
  are convex along the portion of the line $x+y=p$ lying in the
  positive quadrant or neither is.
 \end{lemma}
  \begin{proof}
   This follows directly from the equality $T_{M}(x,y)=T_{M^{*}}(y,x)$.
  \end{proof}

 \begin{lemma}\label{convexity_contraction_deletion}
 Let $M$ be a matroid and $e$ in $M$ be neither a loop nor a coloop. If
 $T_{M\setminus e}$  and $T_{M/e}$ are both convex along the portion
 of the line $x+y=p$ lying in the positive quadrant, then $T_M$ is
 also convex on the same
 domain.
 \end{lemma}
  \begin{proof}
    This follows directly from the deletion-contraction formula
    (\ref{eq:deletion-contraction}) and that the  sum of convex
    functions is  also a convex function.
  \end{proof}

  The following three results deal with the convexity of $T$ for
  some coloopless paving matroids. We use these
 cases as bases for an inductive argument later on.

 \begin{lemma}\label{parallel_convex}
 If M is isomorphic to  the paving  matroid $U_{1,k+1}\oplus U_{0,l}$,
 where $l\geq 0$  and $k\geq 1$, then  $T_M$ is
 convex along the portion of the line $x+y=p$ lying in the positive quadrant.
 \end{lemma}
 \begin{proof}
  We have
    \[ T_M(x,y) = y^l(y^k+\cdots +y+x) = py^l  +\sum_{m=l+2}^{l+k} y^m.\]
  Since $y^m$ is convex for all $m\geq 0$ in the given region and the sum
  of convex
  functions is convex, the result follows.
 \end{proof}

 \begin{lemma}\label{uniform_convex}
  The Tutte polynomial $T_M$ is a convex function in the positive
  quadrant when $M$ is a
  uniform matroid. In particular, $T_M$ is convex along the portion of
  the line $x+y=p$ lying in the positive quadrant
 \end{lemma}
 \begin{proof}
   The Tutte polynomial of a uniform matroid can be computed easily
   using (\ref{eq:expansion}).
   \begin{equation*}
          T_{U_{r,n}} (x,y)= \sum_{i=0}^{r-1}\binom ni(x-1)^{r-i}+
                           \binom nr+
                           \sum_{i=r+1}^{n}\binom ni(y-1)^{i-r}.
    \end{equation*}
  This can be expanded into the following expression, which may also
  be established directly using~(\ref{tutte_activities}).
   \begin{equation*}
          T_{U_{r,n}} (x,y)= \sum_{j=1}^{n-r}\binom{n-j-1}{r-1}y^{j}+
                           \sum_{i=1}^{r}\binom{n-i-1}{n-r-1}x^{i},
    \end{equation*}
  when $0<r<n$, while $T_{U_{n,n}}(x,y)=x^{n}$ and $T_{U_{0,n}}(x,y)=y^{n}$.

 As each term is a convex function we get the result.
 \end{proof}

 \begin{thm}\label{rank_2_convexity}
  If $M$ is a rank-2 loopless and coloopless matroid, then $T_M$ is
  convex along the portion of the line $x+y=p$ lying in the positive quadrant.
 \end{thm}
 \begin{proof}
   If $M$ is isomorphic to the uniform matroid $U_{2,n}$, the result
   follows from applying the previous lemma. Otherwise, $M$ is
   isomorphic to a
   matroid with parallel elements whose simplification is isomorphic
   to $U_{2,n}$.

   If $n\geq 3$ or there is a parallel class of size at
   least 3, we can choose an element $e$ in a non-trivial parallel
   class of $M$ such that $M\setminus e$ does not have a coloop. In
   this case $M/e$ is isomorphic to $U_{1,k+1}\oplus U_{0,l}$,
   where $l\geq 1$  and $k\geq 1$ and $M\setminus e$ is a rank-2
   loopless and coloopless matroid. The result follows from
   Lemma~\ref{parallel_convex}, induction and
   Lemma~\ref{convexity_contraction_deletion}.

   In the last case, the simplification of $M$ is isomorphic to
   $U_{2,2}$ and every element is in a parallel class of size 2. Then
   $M$ is isomorphic to $U_{1,2}\oplus U_{1,2}$. Then, $T_M=(x+y)^2$
   which is convex (in fact is constant) along $x+y=p$ for $p>0$ and
   $0\leq y\leq p$.
 \end{proof}

   In order to establish or main result, we need the following
   structural result about coloopless paving matroids.

   \begin{lemma}\label{lemma_2-stretching}
    Let $M$ be a rank-r coloopless paving matroid. If for every element $e$
    of $M$, $M\setminus e$ has a coloop, then one of the following
    three cases happens.
    \begin{enumerate}
    \item $M$ is isomorphic to $U_{r,r+1}$.\label{2-stretching_i}
    \item $M$ is the 2-stretching of a uniform matroid $N$ and
       $N$ is isomorphic to  $U_{s,s+1}$ or
       $U_{s,s+2}$, for some $s\geq 1$.\label{2-stretching_ii}
    \item $M$ is isomorphic to $U_{1,2}\oplus U_{1,2}$.\label{2-stretching_iii}
    \end{enumerate}
   \end{lemma}
   \begin{proof}
    If $e$ is such that $M\setminus e$ has a coloop $f$, then
    $\{e,f\}$ are in either a  series or form a parallel class. If there is a
    parallel class in a  paving matroid, its rank is either 1 or
    2. Thus, if $\{e,f\}$ are in a parallel class, $M$ is isomorphic
    to  $U_{1,2}\oplus U_{1,2}$ or $U_{1,2}$.

    We can assume that $M$ contains no non-trivial parallel
    classes. Hence every element belongs to a series class of size at
    least two. Suppose that there is a series class containing at
    least three elements $e,f,g$. In
    this case, $M\setminus e$ will have at least 2 coloops. But as $M$ is
    paving all its minors are also paving. Thus, $M\setminus e$, being
    a paving matroid with at least 2 coloops, cannot
    have circuits  and $M\setminus e$ is isomorphic to $U_{r,r}$. In
    this case, we conclude that $M$ is isomorphic to  $U_{r,r+1}$.

    To finish, we suppose that every element in $M$  is in a series
    class of size 2.  In this case, $M$ is the 2-stretching of a
    rank-$s$ matroid $N$ with $m$ elements and $s\geq 1$. $N$ is
    paving because it is a minor of $M$ and it must have circuits as $M$
    is coloopless.

    If the minimal size of a circuit in  $N$ has size $s$, $M$ has a
    circuit of size $2s$.  But the rank of $M$ is $s+m$ as it is the
    2-stretching of $N$. Then $2s\geq s+m$ and $s=m$. In this case,
    $N$ would be isomorphic to $U_{s,s}$ and we arrive at a
    contradiction. Thus, $N$ does not have circuits of size $s$.

    Hence all the circuits of $N$ have size $s+1$ and
    $N$ is uniform. Then, there is a circuit  in $M$ of size $2s+2\geq
    s+m$, and $s+2\geq m\geq s+1$. Thus, $N$ is isomorphic to
    $U_{s,s+1}$ or $U_{s,s+2}$.
   \end{proof}

    \begin{lemma}\label{lemma_2-stretching_convexity}
    Let $M$ be a rank-r coloopless paving matroid. If for every element $e$
    of $M$, $M\setminus e$ has a coloop, then $T_M$ is
    convex along the portion of the line $x+y=p$ lying in the positive quadrant.
   \end{lemma}
   \begin{proof}
    We analyse the cases for $M$ given in the previous lemma. If $M$
    is isomorphic to $U_{r,r+1}$, the result follows from
    Lemma~\ref{uniform_convex}. If $M$ is isomorphic to $U_{1,2}\oplus
    U_{1,2}$ or  $U_{1,2}$, the corresponding Tutte polynomials are
    $(x+y)^2$ and $x+y$, which in both cases are convex.

    If $M$ is the 2-stretching of $U_{s,s+1}$, then $M$ is isomorphic
    to $U_{r, r+1}$ and the result follows from
    Lemma~\ref{uniform_convex}. If $M$ is the 2-stretching of
    $U_{s,s+2}$, then $M^*$ is the 2-thickening of $U_{2,n}$ which
    is a rank-2 matroid and the result follows from
    Theorem~\ref{rank_2_convexity} and
    Lemma~\ref{convexity_duality}.
   \end{proof}

 Finally, we arrive at the main result of this section.

 \begin{thm}\label{convexity_Tutte}
  If $M$ is a  coloopless paving matroid, then $T_M$ is
  convex along the portion of the line $x+y=p$ lying in the positive quadrant.
 \end{thm}
  \begin{proof}
   If $M$ has a loop, then $M$ has rank 1 and it is isomorphic to
   $U_{1,k+1}\oplus U_{0,l}$ with $l,k\geq 1$ and the result follows
   from Lemma~\ref{parallel_convex}.

   Otherwise, every element of $M$ is neither a loop nor a coloop. If there
   is an element $e$ such that $M\setminus e$ has no coloop, then both
   $M/e$ and $M\setminus e$ are coloopless paving matroids and the
   result follows from Lemma~\ref{convexity_contraction_deletion}.

   So, we can assume that  for all $e$, $M\setminus e$ has a coloop. Then
   the result follows from Lemma~\ref{lemma_2-stretching_convexity}.
   \end{proof}

 Hence, subject to an affirmative answer to Welsh's
problem mentioned earlier, we have proved Conjecture~\ref{old} and
Theorem~\ref{convexity_Tutte} for asymptotically almost all matroids.

 Paving matroids are not closed under duality but using
 Lemma~\ref{convexity_duality} we obtain the convexity
 of the Tutte polynomial for a bigger class of matroids.

  \begin{cor}
  If $M$ or $M^{*}$ is a  coloopless paving matroid, then $T_M$ is
  convex  along the portion of the line $x+y=p$ lying in the positive
  quadrant.
 \end{cor}

  By Theorem~\ref{paving_DB_or_UI}, the class of matroids $M$ such
  that either $M$ or $M^{*}$ is a coloopless paving matroid is contained in
  the class of matroids that contains two disjoint bases or whose
  ground set is the union of two bases.   Thus, we have a strengthening
  of Theorem~\ref{main}.

 \begin{cor}\label{strengthening_main}
  If $M$ or $M^{*}$ is a  coloopless paving matroid, then $T_M$ satisfies
  inequality (\ref{main_inequality})  for $a\geq 0$.
 \end{cor}

%%%%%%%%%%%%%%%%%%%%%%%%%%%%%%%%%%%%%%%%%%%%%%%%%%%%%%%%%%%
%%%%%%%%%%  CONJECTURE %%%%%%%%%%%%%%%%%%%%%%%%%%%%%%%%%%%%
%%%%%%%%%%%%%%%%%%%%%%%%%%%%%%%%%%%%%%%%%%%%%%%%%%%%%%%%%%%

\section{The Merino-Welsh  conjecture}
\label{sec:conjecture}

In this section we return to the original Merino--Welsh conjecture
(Conjecture~\ref{old}) and establish that its conclusion holds for
some fairly
specific classes of graphs and matroids. Recall that the conclusion of
the conjecture is certainly not true for all graphs. Taking any graph
and adding a loop and a bridge results in a graph that does not
satisfy~\eqref{eq:MW}. However, the condition on the connectivity may
not be the most natural because if $G$ consists of 2 cycles of length
2 sharing a common vertex, then the graphic matroid $M(G)$ satisfies
(\ref{eq:main}) for all $a\geq 0$. So (\ref{eq:MW}) is satisfied by some
graphs that are not 2-connected.

\subsection{Wheels and whirls}\label{wheels}
 In this subsection we consider  wheels, a well-known class of
 self-dual planar graphs, and whirls, a related class of matroids which
 are also self-dual. The  wheel
 graph  $W_n$ has $n+1$ vertices and $2n$ edges. The  vertices
 $\{1,\cdots, n\}$ form an $n$-cycle while the vertex $0$ is adjacent
 to every vertex in this cycle. The whirl $W^{n}$ is the matroid with
 ground set $E(W^{n})=E(W_{n})$, while the set of bases of $W^{n}$
 consists of the edge set in the $n$-cycle of $W_n$ together with all
 edge sets of spanning trees of $W_n$, see~\cite{Oxley}.

 It is well-known that $\tau(W_n)=L_{2n}-2$, for $n\geq 1$, where
 $L_k$ is the $k$th-Lucas number which is defined recursively by
 $L_1=1$, $L_2=3$ and $L_k=L_{k-1}+L_{k-2}$ for $k\geq 3$. This result
 was proved by Sedl\'a\v cek \cite{Sed} and also by Myers
 \cite{My}. Using the analogy of
 Binet's Fibonacci formula for Lucas numbers we get
$$
   \tau(W_n)=\left( \frac{3+\sqrt{5}}{2}\right)^{n}
             +\left(\frac{3-\sqrt{5}}{2}\right)^{n}-2.
$$
The same formula can be obtained directly by
using equation~(\ref{eq:deletion-contraction}) for $T_{W_n}(1,1)$ and
then solving
the corresponding recurrence relation.

The chromatic polynomial of $W_n$ is known, see \cite{Biggs}, and is
equal to $\chi_{W_n}(x)$ = $x(x-2)^n+(-1)^n x(x-2)$.  Now, applying the
famous result of R. Stanley \cite{Stanley}
 that relates the number of acyclic orientations and the chromatic polynomial,
 namely $\alpha(G)=|\chi_{G}(-1)|$, we get  $\alpha(W_n)= 3^{n}-3$.
 These results together yield the following

\begin{thm}
   For all $n\geq 2$, $\alpha(W_{n})\geq \tau(W_{n})$ and $M(W_{n})$
   satisfies Conjecture~\ref{old}.
\end{thm}

The Tutte polynomials of  $W^{n}$ and  $M(W_{n})$ are related by the
 equality,  $T_{W^{n}}(x,y)=T_{W_{n}}(x,y)-xy+x+y$. Thus, obtain the
 following result

\begin{thm}
   For all $n\geq 2$, $T_{W^{n}}(2,0)\geq T_{W^{n}}(1,1)$ and $W^{n}$
   satisfies equation~(\ref{eq:MW}).
\end{thm}

 \subsection{3-regular graphs with girth at least 5}\label{3-regular}

 For  3-regular graphs  with
 girth at least 5 a
 lower  bound for the number of  acyclic orientations,
\[
  \alpha(G)\geq (2^{3/8} 3^{3/8} 4^{1/8})^n\;,
\]
is given in \cite{KhSc}, where $n$ is the number of vertices of
$G$. On the other hand, the following upper
bound for the number of spanning trees in a 3-regular graph $G$ is
 given in \cite{ChYao}.
\[ \tau(G)\leq
        \frac{2\beta}{3 n}
        e^{\frac{12}{\sqrt{\pi}}
             \left( \frac{1}{\beta}\right)^\frac{5}{2}}
        \left( \frac{4}{\sqrt{3}}\right)^n\;,
\]
 where $\beta=\lceil \ln(n)/\ln(9/8)\rceil$. From the formulae we
 obtain the following

\begin{thm} If $G$ is a 3-regular  graph of girth at least 5, we have
 $\tau(G)< \alpha(G)$ and $M(G)$  satisfies Conjecture~\ref{old}.
\end{thm}

 \subsection{Complete graphs}
\label{complete_bipartite}
 It is natural to check if Conjecture~\ref{old} is  true for complete
 graphs and complete bipartite graphs.

 A classical result of Cayley~\cite{Biggs} states that $\tau(K_n)=n^{n-2}$.
 For $K_3$ we have
 $\alpha(K_3)=6>3=\tau(K_3)$, thus $K_3$ satisfies
 Conjecture~\ref{old}.

 We use the following lemma which has an easy proof, see~\cite{CnMr}.
 \begin{lemma}\label{simplicial}
   If $G$ is a 2-connected graph with a
   vertex $v$ of    degree $d$, then  $(2^d-2)\alpha^*(G-v) \leq
   \alpha^*(G)$.
 \end{lemma}

 We will prove that $\alpha^{*}(K_{n})\geq n^{n-2}$, for $n\geq
 4$. When $n=4$, we have $\alpha^{*}(K_{4})=24>16=\tau(K_{4})$. We
 proceed by induction on $n$.
\begin{equation*}
\begin{split}
 \tau(K_{n+1})= (n+1)^{n-1}&= \left( \frac{n+1}{n}\right)^{n}
                             \left( \frac{n}{n+1}\right)^{2}
                             (n+1)\tau(K_{n})\\
                           &\leq e (n+1) \tau(K_{n})
                             \leq (2^{n}-2) \tau(K_{n})\\
                           &\leq (2^{n}-2) \alpha^{*}(K_{n}).
\end{split}
\end{equation*}
 The last quantity is less than or equal $\alpha^{*}(K_{n+1})$ by the
 previous lemma.

\begin{thm}\label{thm:complete}
 For all $ n\geq 3$, $M(K_{n})$  satisfies  Conjecture~\ref{old}.
\end{thm}

 The technique used for complete graphs can be used to prove the
 Conjecture~\ref{old} in the case of threshold graphs, a type of
 chordal graphs, see~\cite{CnMr}.
 Also in~\cite{CnMr} complete bipartite graphs are considered and the
 authors prove the following

\begin{thm}\label{thm:bipartite}
 For all $m\geq n\geq 2$, $M(K_{n,m})$  satisfies
 Conjecture~\ref{old}.
\end{thm}

  \subsection{Catalan matroids}
   A \emph{Dyck path} of length $2n$ is a path in the plane from (0,0) to
   ($2n$,0), with steps (1,1), called \emph{up-steps}, and (1,-1),
   called \emph{down-steps}. It is well-known that the number of Dyck
   paths of length $2n$ is the Catalan number
   $C_n=\frac{1}{n+1}\binom{2n}{n}$. Each Dyck path $P$ defines an
   \emph{up-step set}, consisting of the integers $i$, $1\leq i\leq 2n$, for
   which the $i$th-step of $P$ is an up-step. The collection of up-step
   sets of all Dyck paths of length $2n$ forms the bases of a matroid $M_{n}$
   over $\{1,2,\ldots, 2n\}$. These matroids are called Catalan
   matroids and have recently been studied extensively,
   see~\cite{AnaBonin} or~\cite{Ardila}.

   We consider the matroids $N_{n}$, $n\geq 2$,  obtained form $M_{n}$ by
   deleting the elements 1 and $2n$. This corresponds to deleting the
   loop and isthmus of $M_{n}$. From the results in~\cite{AnaBonin} it
   follows that the matroid $N_{n}$ is
   self-dual, but not identically self-dual.  An expression for the Tutte
   polynomial of $N_{n}$ follows from Corollary~5.8
   of~\cite{AnaBonin}.
   \[
    T_{N_{n}}(x,y)=\sum_{i,j>0}\frac{i+j-2}{n-1}\binom{2n-i-j-1}{n-i-j+1}x^{i-1}y^{j-1}.
    \]
    After some algebraic manipulations we get a formula for the
    evaluation at (2,0) and (0,2).
    \[ T_{N_{n}}(2,0)=T_{N_{n}}(0,2)=\sum_{k=0}^{m} \frac{k}{m}
     \binom{2m-k-1}{m-k} 2^{k},
    \]
     where $m=n-1$. This quantity equals $\binom{2m}{m}$ by the
     following list of equalities
   \begin{equation*}
     \begin{split}
      \sum_{k=0}^m \frac km \binom {2m-k-1}{m-k} 2^k &=
      \sum_{k=0}^m \bigg( \binom{2m-k-1}{m-1} - \binom{2m-k-1}m \bigg) 2^k\\
      &= \sum_{k=0}^m \sum_{j=0}^k \bigg( \binom{2m-k-1}{m-1} -
      \binom{2m-k-1}m \bigg) \binom {k}{j} \\
      &= \sum_{j=0}^m \sum_{k=j}^m \bigg( \binom{2m-k-1}{m-1} -
      \binom{2m-k-1}m \bigg) \binom {k}{j} \\
      &= \sum_{j=0}^m \bigg( \binom{2m}{m+j} - \binom{2m}{m+j+1} \bigg)\\
      &=\binom{2m}{m}.
     \end{split}
   \end{equation*}
   The key step in the middle uses the convolution identity
   $\sum_{k=0}^{2m-1} \binom {2m-k-1}{q} \binom {k}{j}$   =
   $\binom{2m}{q+j+1}$ that is the basic identity (5.6)
   in~\cite{concrete}. The  value of $T_{N_{n}}(1,1)$ is clearly
   $C_n=\frac{1}{n+1}\binom{2n}{n}$.

   \begin{thm}\label{thm:catalan}
     For all $n\geq 2$, $N_{n}$  satisfies equation~(\ref{eq:MW}).
   \end{thm}

%%%%subsection a new conjecture
Notice that in all of the classes that we have considered, either the
ground set contains two disjoint bases or is the union of two
bases. We therefore propose the following conjecture which is a weaker
form of Conjecture~\ref{old} and may turn out to be more tractable.

\begin{conjecture}\label{new}
 If  $M$ contains two disjoint bases or its ground set  is the union
 of two bases then $\max\{ T_{M}(2,0),\ T_{M}(0,2) \} \geq T_{M}(1,1)$.
\end{conjecture}

 \section{Conclusion and Discussion}

 We have proved that $T_M$ is convex along the portion of the line $x+y=p$
 lying in the positive quadrant, whenever $M$ is a coloopless paving
 matroid.  By Definition~\ref{def:activities},
 $T_M$  is convex along the semilines $y-mx+b$ for $m\geq 0$ and $b\in
 \mathbb{R}$ in the positive quadrant. It is natural to ask for which
 matroids is $T_M$ convex in the positive quadrant?

There is no clear link between convexity of the Tutte polynomial in
the positive quadrant and the classes of matroids that we have
considered. Coloopless paving matroids may or may not have Tutte
polynomials that are convex in the positive quadrant. For example, the
Tutte polynomials of uniform matroids and the graphic matroid $M(K_4)$
are convex in the positive quadrant; on the other hand the Tutte
polynomial $y^{l}(y^k+\ldots +y+x)$ of the
 paving matroid  $U_{1,k+1}\oplus U_{0,l}$, where $l\geq 1$  and
 $k\geq 1$ is not a convex or concave  function. There are also
 non-paving matroids whose
Tutte polynomial is convex, for example $U_{n,n}^2$, for $n\geq 3$,
the 2-thickening of $U_{n,n}$. The Tutte polynomial of this matroid is
$(x+y)^n$ which is clearly convex. Note however that this latter class
of matroids has two disjoint bases.

Establishing the convexity of the Tutte polynomials of matroids within
a given large class seems to be a difficult problem. The Tutte
polynomials of the graphs at the
top of Fig.~1 are convex functions while the Tutte
polynomial of the
graph at the bottom is neither convex nor concave. A similar situation holds
for the matroids in Fig.~2,  the Tutte polynomials
of the two matroids at
the top of the figure are convex functions while the polynomial for the
matroid at the bottom is neither convex nor concave.

   \begin{figure}
    \begin{center}
\psset{unit=0.6}
\pspicture(-5,-5)(15,3)

\qdisk(1,0){0.14}\qdisk(2,-0.1){0.14}\qdisk(3,-0.08){0.14}\qdisk(3,-0.32){0.14}
\qdisk(4,-0.3){0.14}\qdisk(4,-0.54){0.14}\qdisk(4,-0.06){0.14}\qline(1,0)(4,-0.3)

\qdisk(-1,0){0.14}\qdisk(-1,0.24){0.14}\qdisk(-1,-0.24){0.14}\qdisk(-2,0.1){0.14}
\qdisk(-3,0.2){0.14}\qdisk(-4,0.3){0.14}\qline(-1,0)(-4,0.3)

\qdisk(-1.5,-1.5){0.14}\qdisk(-0.5,-1.6){0.14}\qdisk(0.5,-1.7){0.14}\qdisk(1.5,-2.16){0.14}
\qdisk(1.5,-1.92){0.14}\qdisk(1.5,-1.68){0.14}\qdisk(1.5,-1.44){0.14}\qline(1.5,-1.8)(-1.5,-1.5)
\rput(0,-4.5){Figure 1}

\qdisk(11,0){0.14}
\qdisk(14,0){0.14}
\qdisk(12.5,2.6){0.14}
\qline(11,0)(14,0)
\qline(14,0)(12.5,2.6)
\psbezier(12.5,2.6)(13.52,2.03)(14.02,1.17)(14,0)
\psbezier(12.5,2.6)(12.48,1.43)(12.98,0.57)(14,0)
\psbezier(12.5,2.6)(11.48,2.03)(10.98,1.17)(11,0)
\psbezier(12.5,2.6)(12.52,1.43)(12.02,0.57)(11,0)

\qdisk(9,0){0.14}
\qdisk(6,0){0.14}
\qdisk(7.5,2.6){0.14}
\qline(9,0)(6,0)
\qline(6,0)(7.5,2.6)
\psbezier(7.5,2.6)(8.52,2.03)(9.02,1.17)(9,0)
\psbezier(7.5,2.6)(7.48,1.43)(7.98,0.57)(9,0)

\qdisk(10,-1){0.14}
\qdisk(11.5,-3.6){0.14}
\qdisk(8.5,-3.6){0.14}
\qline(10,-1)(11.5,-3.6)
\qline(10,-1)(8.5,-3.6)
\qline(8.5,-3.6)(11.5,-3.6)
\psbezier(10,-1)(11.02,-1.57)(11.52,-2.43)(11.5,-3.6)
\psbezier(10,-1)(9.98,-2.17)(10.48,-3.03)(11.5,-3.6)

\rput(10,-4.5){Figure 2}
\endpspicture
\end{center}

    \end{figure}

  We proved  Conjecture~\ref{old} for some families of
  graphs and matroids.  There are some more families for which the
  conjecture holds: for example Marc Noy
  (private communication) proved that $\tau(G)\leq \alpha(G)$ when $G$
  is a maximal   outerplanar graph,
  using equation~(\ref{eq:deletion-contraction}).

\section{Acknowledgment}
  We thank Bill Jackson for helpful discussions.

\end{document}